\documentclass[12pt]{amsart}
\usepackage{times}
\usepackage{geometry}
\usepackage[utf8]{inputenc}
\usepackage{enumitem,amssymb}
\usepackage{comment}
\usepackage{ragged2e}       
\usepackage[english]{babel}
\usepackage{amsmath}
\usepackage{amsfonts}
\usepackage{mathtools}
\usepackage{graphicx}
\usepackage{mathrsfs}
\usepackage{amsthm}
\usepackage{amssymb}
\usepackage[all]{xy}
\usepackage{tikz-cd}
\usepackage{mathdots}
\usepackage{bbm}
\usepackage{hyperref}
\usepackage{faktor}
\usepackage{mathbbol}
\usepackage{thmtools}
\usepackage{makecell}
\usepackage{caption}
\usepackage{listings}
\usepackage{longtable}
\usepackage{multicol}

\usepackage{tikz}
\usetikzlibrary{calc,patterns}

\newlist{thematic}{itemize}{8}
\setlist[thematic]{label=$\square$}
\usepackage{pifont}

\newtheorem{theorem}{Theorem}[section]
\newtheorem{proposition}[theorem]{Proposition}
\newtheorem{lemma}[theorem]{Lemma}
\newtheorem{corollary}[theorem]{Corollary}
\theoremstyle{definition}
\newtheorem{remark}[theorem]{Remark}
\newtheorem{example}[theorem]{Example}

\DeclareMathOperator{\Hom}{Hom}
\DeclareMathOperator{\Ext}{Ext}

\DeclareMathOperator{\Coh}{Coh}
\DeclareMathOperator{\ch}{ch}
\DeclareMathOperator{\Pic}{Pic}

\usepackage{palatino}			
\usepackage[T1]{fontenc}		
\usepackage[utf8]{inputenc}	

\usepackage{mathpazo}

\title{Bridgeland walls destabilizing one-dimensional space sheaves
}
\author{Daniel Bernal}
\address{IMECC - UNICAMP \\ Departamento de Matem\'atica \\
Rua S\'ergio  Buarque de Holanda, 651\\ 13083-970 Campinas-SP, Brazil}
\address{University of Edinburgh \\
Old College \\
South Bridge \\
Edinburgh \\
EH8 9YL}
\email{o235397@dac.unicamp.br}
\email{s2906019@ed.ac.uk}

\author[Cristian Martinez]{Cristian Martinez}
\address{Departamento de Matemáticas\\
Universidad de los Andes \\
Carrera 1 \# 18A-12\\
111711\\ Bogot\'a\\
Colombia}
\email{crm.martineze@uniandes.edu.co}

\date{}

\begin{document}

 \begin{abstract}
 Following the setup proposed in \cite{vertical2023} and \cite{instantons2023} in the case of the projective space, we study some numerical and actual Bridgeland walls for the (twisted) Chern character $v=(-R,0,D,0)$ in certain half-plane of stability conditions, where walls are nested and finite. We give two bounds for the largest numerical wall that may appear. When $R=0$, these bounds in particular produce the first known bounds for the Gieseker chamber in the case of a threefold. We also study the cases $R=0$ and $D=3,4$ using a small algorithm in Python.
\end{abstract}
\maketitle 

\setcounter{tocdepth}{1}

\tableofcontents

\section{Introduction}

Since the introduction of stability conditions by Bridgeland \cite{B07} and the first class of nontrivial examples on surfaces in \cite{B08} and \cite{AB}, they have become a powerful tool to study the geometry of moduli spaces of semistable sheaves. In particular, interpreting known moduli spaces and their birational transformations as moduli spaces of semistable objects for some stability conditions has opened up a new world of possibilities (see \cite{mmp13}, \cite{BMW}, \cite{BM14a}, \cite{martinez2017duality}, \cite{instantons2023} and \cite{DTPTjlmm} for examples of applications of these so called wall-crossing techniques).

\vspace{0.3cm}
\noindent
On surfaces, most of the applications listed above rely on the fact that it is possible to find a chamber in the stability manifold where the stable objects of a given Chern character are precisely the Gieseker semistable sheaves. Furthermore, precise bounds for this ``Gieseker chamber'' are given in \cite{lomartinez23}. However, in the case of threefolds, the only general construction for a family of stability conditions, due to Bayer, Macrì and Toda \cite{BMT}, is still conjectural and at the moment of writing only proven for a handful of cases: Fano threefolds of Picard rank 1 \cite{Li,M-p3,S-q3}, more general Fano threefolds \cite{BMSZ}, abelian threefolds \cite{MP,BMS}, quintic threefolds \cite{Li2}, general weighted hypersurfaces in either of the weighted projective spaces $\mathbb{P}(1,1,1,1,2)$ or $\mathbb{P}(1,1,1,1,4)$ \cite{KosDoubleTriple}, some product threefolds \cite{KOSproducts}, threefolds with nef tangent bundles \cite{KOSnef} and Calabi--Yau threefolds obtained as a complete intersection of quadric and quartic hypersurfaces in $\mathbb{P}^5$ \cite{LiuBGICY3}. In these cases, and when the Picard rank is 1, we have for every $\beta\in\mathbb{R}$, a 2-parameter family of stability conditions $\{\sigma_{\beta,\alpha,s}\}_{\alpha,s>0}$, each defined over a tilt of the tilted heart $\Coh^{\beta}(X)$.

\vspace{0.3cm}
\noindent
The double-tilted nature of the categories supporting stability conditions on the threefolds listed above makes finding the Gieseker chamber extremely difficult. Nevertheless, when $v$ is the Chern character of a Gieseker semistable sheaf of positive rank and the threefold has Picard rank 1, it was shown in \cite{DTPTjlmm} that for every $s>1/3$ there is a value of $\alpha>0$ (depending only on $v$) such that for every $\beta\ll 0$ the Bridgeland semistable objects of Chern character $v$ are precisely the Gieseker semistable sheaves. The same result when $X$ is abelian was obtained in \cite{Oberdieck}. When $v$ is the Chern character of a 1-dimensional Gieseker semistable sheaf, it was shown in \cite{vertical2023} that there is a specific value of $\beta$ (depending only on $v$) such that the Bridgeland semistable objects of Chern character $v$ for $\alpha\gg 0$ are precisely the Gieseker semistable sheaves. However, a bound for the Gieseker chamber (similar to the one in \cite{lomartinez23} for the case of surfaces) is still unknown.    

\vspace{0.3cm}
\noindent
In this paper, we show the existence of actual walls for objects with $\beta$-twisted Chern character $v=(-R,0,D,0)$ (where $R\geq 0,D>0$) in the $(\alpha,s)$-plane parameterizing the family $\{\sigma_{\beta,\alpha,s}\}_{\alpha,s>0}$ of Bridgeland stability conditions on $\mathbb{P}^3$. It was shown in \cite[Section 5]{vertical2023} that on this slice of stability conditions the walls are disjoint, finite, and can be parameterized by a value $\alpha_0>0$ (that can be written in terms of the third and first Chern characters of the destabilizing subobject).

\vspace{0.3cm}
\noindent
Moreover, we give bounds using similar methods as in \cite{schmidt3folds20} and \cite{lomartinez23} for the Gieseker chamber, thus giving the first of such bounds for the case of threefolds. More precisely, we prove the following:
\begin{theorem}
\label{thm1.1}
    Let $X=\mathbb{P}^3$ and $\ch^\beta= (0,0,D,0)$. Then 
    \begin{enumerate}
        \item For $\alpha_0>2D$ there are no numerical nor actual walls. In particular, this region is contained in the Gieseker chamber.
        \item If $\beta=0$ there are at least three actual walls  and one numerical wall in the stability manifold $\operatorname{Stab}(\mathbb{P}^3)$ and we have a sharper bound for the Gieseker chamber. Namely, for $\alpha_0>D$ there are no numerical (and thus actual) walls. 
    \end{enumerate}
\end{theorem}
\noindent
Using this result, we then turn to the question of identifying walls in the cases when the asymptotic moduli space ($\alpha \gg 0$) is known. Using the descriptions of the Gieseker moduli given in  \cite{jardim2017moduli}, \cite{cubics2004}, and \cite{quartics2016}, we analyze in detail the cases $D=3,4$ and $\beta\neq 0$, where the asymptotic Bridgeland moduli spaces are indeed moduli spaces of 1-dimensional Gieseker semistable sheaves (see \cite[Proposition 5.3]{vertical2023}).

\vspace{0.3cm}
\noindent
For doing so, we work out the numerical conditions provided in \cite{vertical2023} (see Section 3.1) for the possible Chern characters of a destabilizing subobject. The numerical walls coming from these conditions are called \textit{pseudo-walls}, see \cite[Definition 4.19]{walls2022}. 

\vspace{0.3cm}
\noindent
When studying the cases with $\beta\neq 0$ the level of difficulty in finding solutions for our numerical conditions increases greatly. However, thanks to the implementation of a small algorithm in Python, we are able to prove the following:

\begin{theorem}\label{thm1.2}
    Let $X=\mathbb{P}^3$ and $\ch^{\beta}=(0,0,D,0)$ with $D=3,4$, then the boundary of the Gieseker chamber in the $(\alpha,s)$-slice destabilizes plane sheaves. Moreover, the corresponding destabilizing sequences are pushforwards of destabilizing sequences for tilt stability on $\mathbb{P}^2$.
\end{theorem}
\noindent
Section \ref{sec:prelim} contains some preliminaries on the Bayer, Macrì and Toda construction of Bridgeland stability conditions on threefolds along with the numerical constraints for the existence of walls given in \cite{vertical2023} and some results on the moduli space of Gieseker semistable sheaves with Hilbert polynomial $Dt$ given in \cite{jardim2017moduli}. In Section \ref{sec:wallsandnumcond}, we study some actual walls that appear when $\beta=0$ and give the numerical conditions for pseudo-walls in the $\beta \neq 0$ case. In Section \ref{sec:Giesekerwall}, we give bounds for the maximum numerical wall in the $(\alpha,s)$-slice of $\operatorname{Stab}(\mathbb{P}^3)$, which for the special case $R=0$ bounds the Gieseker chamber. In Section \ref{sec:examples}, we use the results from Sections \ref{sec:wallsandnumcond} and \ref{sec:Giesekerwall} to study the cases of $D=3$ and $D=4$. Lastly, in the Appendices we give the \href{https://github.com/daniel-bernalm/walls_computations}{code} in Python used to compute the numerical walls for the cases $\beta=0$ or $\beta^{-1}\in\mathbb{Z}$.

\subsection*{Acknowledgments}
DB would like to thank his advisor Marcos Jardim for the suggestion on the topics, guidance on the manuscript and the support offered. This project was partially developed at the University of Edinburgh to which DB is grateful for the hospitality and support. The authors would also like to express their gratitude to Antony Maciocia and Jason Lo for the comments and suggestions on earlier versions of this text. DB is supported by the Coordenação de Aperfeiçoamento de Pessoal de Nível Superior - Brasil (CAPES) - Finance Code 001 - Process Number 88887.015001/2024-00, by the CAPES-PRINT Program - Call no. 41/2017 - Process Number 88887.975395/2024-00 and by the CAPES-PDSE Program - Process Number 88881.128062/2025-01.


\section{Preliminaries}\label{sec:prelim}
\subsection{Notation} Other than specified, we will use the following notation:
\begin{itemize}
    \item $X$ will denote a threefold of Picard rank 1 for which the Bayer-Macrì-Toda construction of Bridgeland stability conditions works, e.g., $X$ is Fano or abelian, and $H$ will denote the ample generator of $\Pic(X)$.
    \item $D^b(X)$ will denote the bounded derived category of coherent sheaves on $X$.
    \item For $\beta\in\mathbb{R}$, the twisted Chern character of an object $A\in D^b(X)$ is
    $$
    \ch^{\beta}(A):=\exp(-\beta H)\ch(A)=\ch_0^{\beta}(A)+\ch_1^{\beta}(A)H+\ch_2^{\beta}(A)H^2+\ch_3^{\beta}(A)H^3.
    $$
    We then identify $\ch^{\beta}(A)$ with the numerical vector $(\ch_0^{\beta}(A),\ch_1^{\beta}(A),\ch_2^{\beta}(A),\ch_3^{\beta}(A))$. A short computation shows that
    \begin{align*}
        \ch_0^{\beta}(A)&=\ch_0(A),\\
        \ch_1^{\beta}(A)&=\frac{\ch_1(A)H^2}{H^3}-\beta\ch_0(A)\\
        \ch_2^{\beta}(A)&=\frac{\ch_2(A)H}{H^3}-\beta \frac{\ch_1(A)H^2}{H^3}+\frac{\beta^2}{2}\ch_0(A)\\
        \ch_3^{\beta}(A)&=\frac{\ch_3(A)}{H^3}-\beta\frac{\ch_2(A)H}{H^3}+\frac{\beta^2}{2}\frac{\ch_1(A)H^2}{H^3}-\frac{\beta^3}{6}\ch_0(A).
    \end{align*}
    \item If $\mathcal{A}$ is the heart of a bounded t-structure on $D^b(X)$, we denote the cohomologies of an object $E^{\bullet}\in D^b(X)$ with respect to $\mathcal{A}$ by $\mathcal{H}^j_{\mathcal{A}}(E^{\bullet})$, while the cohomologies with respect to the standard t-structure $\Coh(X)\subset D^b(X)$ will be denoted by $\mathcal{H}^j(E^{\bullet})$.
\end{itemize}
\subsection{The double-tilt construction on threefolds}
We assume that the reader is familiar with the basic concepts of Bridgeland stability conditions, as in the topics covered by \cite{macrischmidt19}. For the case of threefolds, the double-tilt construction of Bridgeland stability conditions due to Bayer, Macrì and Toda \cite{BMT} is as follows: 

\begin{enumerate}
    \item Let $\beta\in\mathbb{R}$ and start by considering the \textbf{(twisted) Mumford slope}
$$
\mu_{\beta} (A) =
\begin{dcases}
\frac{\ch^{\beta}_1(A)}{\ch^{\beta}_0 (A)} & \textnormal{ if } \ch^{\beta}_0(A)\neq 0 \\
+\infty & \textnormal{ if } \ch^{\beta}_0(A) = 0; 
\end{dcases}
$$
because of the Harder-Narasimhan property, the subcategories
\begin{gather*}
\mathcal{T}_\beta =  \{
E \in \Coh(X) \, | \, \mu_\beta (Q) > 0 \textnormal{ for every quotient } E \twoheadrightarrow Q 
\} \\
\mathcal{F}_\beta =  \{
E \in \Coh(X) \, | \, \mu_\beta (F) \leq 0 \textnormal{ for every subsheaf } F \hookrightarrow E 
\}    
\end{gather*}
form a torsion pair. Thus, the tilted category $\Coh^{\beta} (X) := \langle \mathcal{F}_{\beta}[1], \mathcal{T}_{\beta} \rangle$ is the heart of a bounded t-structure on $D^b(X)$.

\item For $\alpha>0$, in $\Coh^{\beta} (X)$ we can define the \textbf{tilt-slope} function
$$
\nu_{\beta,\alpha} (A):= 
\begin{dcases}
    \frac{\ch^{\beta}_2 (A) - \frac{\alpha^2}{2} \ch^{\beta}_0(A)}{ \ch^{\beta}_1(A)} & \textnormal{ if } \ch^{\beta}_1(A) \neq 0 \\
    + \infty & \textnormal{ if } \ch^{\beta}_1(A) = 0,
\end{dcases}
$$
which again turns out to have the Harder-Narasimhan property, i.e., each object in $\Coh^{\beta}(X)$ has a unique filtration whose factors are tilt-semistable and of decreasing tilt-slopes. Therefore, the following subcategories
\begin{gather*}
    \mathcal{T}_{\beta,\alpha} =  \{
    E \in \Coh^{\beta}(X) \, | \, \nu_{\beta,\alpha} (Q) > 0 \textnormal{ for every quotient } E \twoheadrightarrow Q \textnormal{ in } \Coh^{\beta}(X)
    \} \\
    \mathcal{F}_{\beta,\alpha} =  \{
    E \in \Coh^{\beta}(X) \, | \, \nu_{\beta,\alpha} (F) \leq 0 \textnormal{ for every subobject } F \hookrightarrow E \textnormal{ in } \Coh^{\beta}(X) 
    \}    
\end{gather*}
form again a torsion pair. Thus, the tilted category $\mathcal{A}^{\beta,\alpha}= \langle \mathcal{F}_{\beta,\alpha}[1], \mathcal{T}_{\beta, \alpha} \rangle $ is again the heart of bounded t-structure on $D^b(X)$.

\item In all the threefolds listed in the introduction (e.g., Fano or abelian), the heart $\mathcal{A}^{\beta,\alpha}$ supports a $1$-parameter family of central charges
$$
Z_{\beta,\alpha,s}(A)=-\left( \ch_3^\beta (A) -\left(s+\frac{1}{6}\right) \ch_1^\beta(A) \right) + \sqrt{-1} \left( \ch^\beta_2 (A) - \frac{\alpha^2}{2} \ch^{\beta}_0 (A) \right),\ s>0,
$$
with corresponding Bridgeland slopes
$$
\lambda_{\beta,\alpha, s} (A):= \frac{
    \ch_3^\beta (A) -\left(s+\frac{1}{6}\right) \alpha^2 \ch_1^\beta(A)
}{
    \ch^\beta_2 (A) - \frac{\alpha^2}{2} \ch^{\beta}_0 (A).
}
$$
Moreover, these central charges define honest Bridgeland stability conditions as they satisfy the support property with respect to the quadratic form
$$
Q_{\beta,\alpha,K}(A)=\ch^{\beta}(A)B_{\alpha,K}\ch^{\beta}(A)^t, \ \ \text{where}\ \ B_{\alpha,K}=\begin{pmatrix}
    0 & 0 &-K\alpha^2 & 0\\
    0 & K\alpha^2 & 0 & -3\\
    -K\alpha^2 & 0 & 4 &0\\
    0 & -3 & 0 & 0
\end{pmatrix}
$$
and $K\in [1,6s+1)$. 
\end{enumerate}
\vspace{0.3cm}
\subsection{The preferred slice and numerical constraints} Fix $\beta \in \mathbb{R}$ and a twisted Chern character vector $v=\ch^{\beta}=(-R,0,D,0)$, with $R\geq 0$ and $D>0$. In the 2-dimensional cone of Bridgeland stability conditions $\sigma_{\beta,\alpha,s}= \left(
\mathcal{A}^{\beta,\alpha}, Z_{\beta,\alpha,s}\right)$, with $\alpha, s>0$, an \textbf{actual Bridgeland wall} for $v$ is given by a short exact sequence in $\mathcal{A}^{\beta,\alpha}$
\[
\begin{tikzcd}
    0 \arrow{r} & A \arrow{r} & E \arrow{r} & B \arrow{r} & 0,
\end{tikzcd}
\]
such that $\ch^{\beta}(E)=v$ and  $\lambda_{\beta,\alpha,s}(A)= \lambda_{\beta,\alpha,s}(E)$, which is equivalent to having
\begin{equation}
    \label{eq:1}
    \lambda_{\beta,\alpha,s}(A)= 0
\end{equation}
since $\ch_3^\beta(E)=\ch^{\beta}_1(E)=0$, while a \textbf{numerical wall} consists uniquely of a solution of the equation (\ref{eq:1}). By a variant of Bertram's Lemma (see \cite[Lemma 5.1]{vertical2023}) we know that if there is a destabilizing subobject for a point on a wall, then that subobject destabilizes the same object along every $(\alpha,s)$ on that wall. Now, using the formula for the Bridgeland slope one gets that in the $(\alpha,s)$-slice of stability conditions, all numerical walls for the Chern character $v$ have the form 
\[
\left(s+\frac{1}{6}\right) \alpha^2 = \frac{\ch_3^\beta(A)}{\ch_1^\beta (A)} =: \frac{\alpha_0^2}{6}.
\] 
Thus, all numerical walls are disjoint hyperbolas, each intersecting the $\alpha$-axis at a point $\alpha_0$. For this reason, we may refer to a numerical wall by its corresponding value of $\alpha_0$. In \cite{vertical2023} was also shown that if $A$ is a destabilizing subobject of $E$ with $\ch^{\beta}(A)=(r,c,d,e)$ and $c>0$, then $r,c,d$ and $e$ satisfy the following \textbf{numerical conditions}:
\begin{gather}
\label{num1}
    0<2d<2D, \tag{N1} \\
   \label{num2} 0<c(6e) \leq \min\left\{ 4d^2, 4(D-d)^2 \right\}, \tag{N2} \\
   \label{num3} - \frac{c(2D-2d)}{6e}-R \leq r \leq \frac{c(2d)}{6e}. \tag{N3}
\end{gather}

\vspace{0.3cm}
\noindent
These bounds are obtained after noticing that at all points of the wall produced by the subobject $A$ (with $s\gg 0$ in the $(\alpha,s)$-slice) we must have
$$
0\leq Q_{\beta,\alpha,s}(A)\leq Q_{\beta,\alpha,s}(E).
$$
In the special case where $X=\mathbb{P}^3$ and $\beta=0$, in addition to the bounds \eqref{num1}, \eqref{num2} and \eqref{num3}, we also have the \textbf{integral conditions}
\begin{gather}
\label{int1}
d-\frac{c^2}{2} \in \mathbb{Z}, \tag{I1} \\
\label{int2} 2e-cd+\frac{c^3}{6} \in \mathbb{Z}, \tag{I2} \\
\label{int3} e-\frac{c}{6} \in \mathbb{Z}, \tag{I3}
\end{gather}  
coming from the fact that $c_1(A),\ c_2(A)$ and $\chi(A)$ are integers. In this particular case, it was also obtained that there is always an actual wall $\alpha_0=1$, originated by the subobject $\mathcal{O}_{\mathbb{P}^3}(1)$ which destabilizes every object, thus called the \textbf{killing wall}. This, along with the numerical and integral conditions was used to find all the actual walls for the Chern character $v=(0,0,3,0)$ and to give a list of possible destabilizing sequences (see Figure \ref{fig:1} and Table \ref{table:1}).

\begin{figure}
\centering
  \includegraphics[width=.5\textwidth]{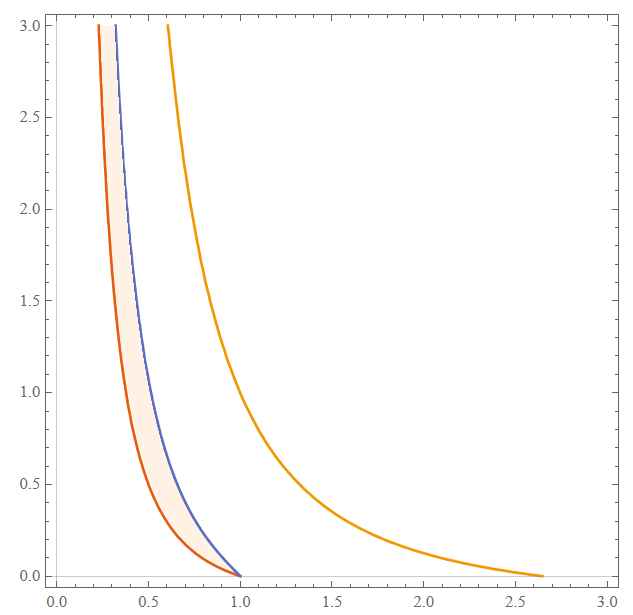}
  \captionof{figure}{The actual walls $\alpha_0=1,\sqrt{7}$ for $v=(0,0,3,0)$ in orange, and the quiver region shadowed in light orange.}
  \label{fig:1}
\end{figure}

\subsection{Moduli spaces of 1-dimensional space sheaves} Lastly, let us denote the Gieseker moduli space of space sheaves with Hilbert polynomial $p(t)$ by $\mathcal{M}_{p(t)}$. When $p(t)=dt$ with $d>1$, according to \cite{jardim2017moduli}, there are some distinguished subsets: 
\begin{itemize}
    \item $\mathcal{P}_d$: consisting of planar sheaves, which is closed, irreducible and projective,
    \item $\mathcal{R}_d$: the closure of the locus of sheaves supported on rational curves of degree $d$, and 
    \item $\mathcal{E}_d$: the closure of the locus of sheaves supported on elliptic curves of degree $d$.
\end{itemize}
It was also shown that it is possible to describe the components of $\mathcal{M}_{dt}$ in terms of decreasing integer partitions of $d= (d_1,d_2, \dots, d_l)$ with $d_1 \geq \cdots \geq d_l$.  

\begin{lemma}{\cite[Lemma 13]{jardim2017moduli}}
    The irreducible components of $\mathcal{M}_{dt}$ are exhausted by the irreducible components of the union
    \[
    \bigcup_{(d_1,\ldots, d_l)} \overline{\mathcal{M}}_{d_1,\ldots, d_l}
    \]
    where $\mathcal{M}_{d_1,\ldots, d_l} \subset \mathcal{M}_{dt}$ denotes the locally closed subset of $S$-equivalence classes $[F_1 \oplus \cdots \oplus F_l]$ with $F_i \in \mathcal{M}_{d_it}$ stable. 
\end{lemma}

\noindent
Even more, within the proof of this lemma, it is proven that each irreducible component of $\overline{\mathcal{M}}_{d_1,\dots,d_l}$ is birational to a symmetric product of irreducible components $\mathcal{X}_k\subset \mathcal{M}_{(d-k)t}$ for $k=1,\dots,l$.


\section{General walls and numerical conditions}\label{sec:wallsandnumcond}

This section has the intention of describing a few actual walls that always appear for the case $v=(0,0,D,0)$ with $\beta=0$, based on the description of the moduli space $\mathcal{M}_{Dt}$. We will also discuss the integral and numerical conditions for a destabilizing subobject, as in Section \ref{sec:prelim}, for the case when $\beta=k^{-1}$ for some $k\in\mathbb{N}$. 

\begin{lemma}
\label{lemma3.1}
    If $\alpha_0$ a numerical wall for the Chern character $v= (-R,0,D,0)$, then $\alpha_0$ is a numerical wall for the Chern character $(-R,0,D+1,0)$. Moreover, if $\beta=0$ and $\alpha_0$ is an actual wall destabilizing a 1-dimensional Gieseker semistable sheaf of Chern character $v=(0,0,D,0)$, then $\alpha_0$ is also an actual wall for the Chern character $(0,0,D+1,0)$.
\end{lemma}

\begin{proof}
    The numerical case follows as the numerical conditions \eqref{num1}, \eqref{num2} and \eqref{num3} will still hold when we increase $D$. For the actual case, suppose that 
    $
    0\rightarrow A\rightarrow F \rightarrow B\rightarrow 0
    $ is a short exact sequence in $\mathcal{A}^{0,\alpha}$ destabilizing a 1-dimensional Gieseker semistable sheaf $F$ with $\ch(F) = (0,0,D,0)$ and producing the wall
    \begin{equation}\label{walllemma}
    \left(s+\frac{1}{6}\right)\alpha^2=\frac{\ch_3(A)}{\ch_1(A)}=\frac{\ch_3(B)}{\ch_1(B)}.
    \end{equation}
    \noindent
     Since we have an inclusion $\overline{\mathcal{M}}_{D,1}\subset \mathcal{M}_{(D+1)t}$, then every 1-dimensional Gieseker semistable sheaf in the S-equivalence class $[F\oplus \mathcal{O}_L(1)]$ will have either $A$ as subobject or $B$ as a quotient in $\mathcal{A}^{0,\alpha}$. Thus, \eqref{walllemma} will also be the equation of an actual wall for the Chern character $(0,0,D+1,0)$.
\end{proof}
\noindent
We now proceed with our first main result:
\begin{proposition}
\label{prop3.3}
    Let $\ch = (0,0,D,0)$. Then for $D\geq 4$ there are at least three actual walls in the $(\alpha,s)$-slice of $\operatorname{Stab}(\mathbb{P}^3)$, namely $\alpha_0 = 1, 2 \textnormal{ and } \sqrt{7}$. We will call them the killing, elliptical and planar wall, respectively. 
\end{proposition}

\begin{proof}
\textbf{The killing wall}. The actualness of the wall with $\alpha_0=1$ was proved in \cite{vertical2023}. Furthermore, the authors also showed that there are no Bridgeland semistable objects of Chern character $(0,0,D,0)$ in the region
$$
\left(s+\frac{1}{6}\right)\alpha^2<\frac{1}{6}.
$$
\textbf{The planar wall}. The Chern character associated to this numerical wall is of the form $\left(r,1,\frac{3}{2},\frac{7}{6} \right)$, with $r=0$ being always a possibility. In this particular case, if $H\subset \mathbb{P}^3$ is a plane, then $\mathcal{O}_H (2)$ has the correct destabilizing Chern character. As observed in \cite[Example 6.14]{vertical2023}, this wall is actual when $D=3$. In fact, it destabilizes the sheaves $\mathcal{O}_C(2)$ for $C$ a cubic plane curve, which fit into an exact sequence in $\mathcal{A}^{0,\alpha}$:
$$
0\longrightarrow \mathcal{O}_H(2)\longrightarrow \mathcal{O}_C(2)\longrightarrow \mathcal{O}_H(-1)[1]\longrightarrow 0.
$$
Then, for $D>3$ the result follows from Lemma \ref{lemma3.1}.
\\
  \noindent      
 \textbf{The elliptical wall}. This is the wall $\alpha_0=2$ associated to the destabilizing Chern character $\left(r,2,2,\frac{4}{3}\right)$. To see that this wall is actual, take $D=4$ and let $\mathcal{E}$ be an elliptic curve of degree $4$. Then $\mathcal{O}(2)$ is a destabilizing subobject of $\mathcal{O}_\mathcal{E}(2)$ with Chern character of this type (and $r=0$). More specifically, we get the exact triangle
 $$
 \mathcal{O}(2)\longrightarrow \mathcal{O}_{\mathcal{E}}(2)\longrightarrow \mathcal{I}_{\mathcal{E}}(2)[1]\longrightarrow \mathcal{O}(2)[1],
 $$
which is in fact a short exact sequence in the category $\mathcal{A}^{0,\alpha}$, proving that the wall is actual. To see this, notice that since $\mathcal{E}$ is irreducible then $\mathcal{I}_{\mathcal{E}}(2)$ is Mumford semistable with $\mu_H(\mathcal{I}_{\mathcal{E}}(2))=2>0$ and therefore $\mathcal{I}_{\mathcal{E}}(2)\in\Coh^0(\mathbb{P}^3)$. Moreover, using the estimate for the radius of tilt-walls \cite[Lemma 7.2]{macrischmidt19} and that $\mathcal{I}_{\mathcal{E}}(2)$ is tilt-semistable for $\alpha\gg 0$ since it is 2-Gieseker semistable, we obtain that $\mathcal{I}_{\mathcal{E}}(2)$ is $\nu_{0,\alpha}$-semistable for 
$$
\alpha^2\geq\frac{\Delta(\mathcal{I}_{\mathcal{E}}(2))}{4}=2.
$$
Thus, $\mathcal{I}_{\mathcal{E}}(2)[1]\in \mathcal{A}^{0,\alpha}$ at least for $\alpha\geq \sqrt{2}$,  since $\nu_{0,\alpha}(\mathcal{I}_{\mathcal{E}}(2))=-1-\frac{\alpha^2}{4}<0$.
\end{proof}

\begin{example}\label{selfdualwall}
    It is easy to verify that for $\beta=0$, the Chern character of the object $\mathcal{O}_H\left(\lceil\frac{D}{2}\rceil\right)$, always satisfies the numerical conditions and so gives a numerical wall. A possible  destabilizing sequence could be
    $$
    0\longrightarrow \mathcal{O}_H\left(\left\lceil\frac{D}{2}\right\rceil\right) \longrightarrow E\longrightarrow \mathcal{O}_H\left(\left\lceil\frac{D}{2}\right\rceil\right)^{\vee}[2]\longrightarrow 0 .
    $$
    It is possible to show that in fact there are always objects $E$ corresponding to nonsplit extensions, however, whether this object is Bridgeland semistable depends on the semistability of $\mathcal{O}_H\left(\lceil\frac{D}{2}\rceil\right)$ on the corresponding numerical wall. This is highly nontrivial, even in the case of tilt stability for surfaces, the pushforward of a line bundle on a curve might be Bridgeland unstable for some values of the stability parameters. When $D$ is odd, these extensions correspond to the (twisted) structure sheaves of plane curves of degree $D$, which fit into exact sequences
    $$
    0\longrightarrow \mathcal{O}_H\left(\frac{D+1}{2}\right)\longrightarrow \mathcal{O}_C\left(\frac{D+1}{2}\right)\longrightarrow \mathcal{O}_H\left(\frac{1-D}{2}\right)[1]\longrightarrow 0.
    $$
    This destabilizing sequence corresponds to the outermost  tilt wall for the Chern character $(0,D,\frac{D}{2})$ on $\mathbb{P}^2$ as studied in \cite{martinez2017duality}.
\end{example}
\begin{remark}\label{p2p3correspondence}
    As suggested in Example \ref{selfdualwall}, one way of guessing possible destabilizing sequences for the Chern character $(0,0,D,0)$ in the $(\alpha,s)$-slice would be to pushforward destabilizing sequences for tilt stability on $\mathbb{P}^2$. This is justified because if $H\subset \mathbb{P}^3$ is a hyperplane and $\iota \colon \mathbb{P}^2\hookrightarrow \mathbb{P}^3$ is the corresponding closed immersion, then given $\mathcal{F}\in D^b(\mathbb{P}^2)$ with $\ch(\mathcal{F})=(r,c,d)$, we can compute $\ch(\iota_*\mathcal{F})=(0,r,c-\frac{r}{2},d-\frac{c}{2}+\frac{r}{6})$. With this in mind, under the choices
    $$
    w=(0,c,d),\ \ \ v=\iota_*(w)=(0,0,c,d-\frac{c}{2}),\ \ \ \bar{s}=\frac{d}{c},\ \ \ \bar{\beta}=\frac{d}{c}-\frac{1}{2},\ \ \ t^2=\left(2s+\frac{1}{3}\right)\alpha^2-\frac{1}{12},
    $$
    we get that 
    $$
    \nu_{\bar{s},t}(w)=\lambda_{\bar{\beta},\alpha,s}(v)=0,\ \text{and}\ \ \nu_{\bar{s},t}(A)=0\ \ \text{if and only if}\ \ \lambda_{\bar{\beta},\alpha,s}(\iota_*(A))=0.
    $$
    Thus, every wall destabilizing a 1-dimensional Gieseker semistable sheaf $\mathcal{F}\in \Coh(\mathbb{P}^2)$ of Chern character $w$ with respect to $\nu_{\bar{s},t}$-stability produces a numerical wall for the Chern character $v$. Moreover, this numerical wall destabilizes $\iota_*\mathcal{F}$ provided that the pushforward of the destabilizing sequence in $\Coh^{\bar{s}}(\mathbb{P}^2)$ becomes a short exact sequence in $\mathcal{A}^{\bar{\beta},\alpha}$.
\end{remark}

\subsection{Numerical walls for \texorpdfstring{$\beta\neq 0$}{beta not zero}}

Here we will describe the inequalities that arise from the equations (\ref{num1}),(\ref{num2}),(\ref{num3}) and (\ref{int1}),(\ref{int2}),(\ref{int3}) when $\beta \neq 0$ and for the particular case of $\beta=\frac{1}{k}$. More specifically, we would like to find all the integral solutions to such inequalities. Setting $\ch^{\beta}=(r,c,d,e)$, we obtain
$$
\begin{bmatrix}
    r\\ c\\ d\\ e
\end{bmatrix}
=
\begin{bmatrix}
    1 & 0 & 0 & 0\\
    -\beta & 1 & 0 & 0 \\
    {\displaystyle \frac{\beta^2}{2}} & -\beta & 1 & 0 \\
    {\displaystyle -\frac{\beta^3}{6}} & {\displaystyle \frac{\beta^2}{2}} & -\beta & 1
\end{bmatrix}\begin{bmatrix}
    \ch_0\\ \ch_1 \\ \ch_2\\ \ch_3
\end{bmatrix}.
$$ 
With this in mind, our first filter becomes:
\begin{align*}
    \ch_2 - \frac{\ch_1^2}{2} &\in \mathbb{Z} \\
    r\frac{\beta^2}{2} + c\beta + d - \frac{(c+\beta r)^2}{2} &\in \mathbb{Z} \\
    \label{eqibeta1}
    d- \frac{c^2}{2} + \beta (1-r) \left( \frac{\beta}{2} r + c \right) & \in \mathbb{Z} \tag{I$\beta$1};
\end{align*}
our second filter becomes:
\begin{align*}
    \ch_3 - \frac{\ch_1}{6} & \in \mathbb{Z} \\
    \label{eqibeta2}
    e-\frac{c}{6}+\beta \left( d-\frac{r}{6} \right) + \frac{\beta^2}{2}c + \frac{\beta^3}{6} r & \in \mathbb{Z} \tag{I$\beta$2};
\end{align*}
and the third one:
\begin{align*}
    2\ch_3 - \ch_1 \ch_2 + \frac{\ch_1^3}{6} & \in \mathbb{Z} \\
    \label{eqibeta3}
    \begin{split}
        2e-cd+ \frac{c^3}{6} + \beta \left( d(2-r) + c^2 (3r-1) \right) \\
        + \frac{\beta^2}{2} c(2+r(r-3)) + \frac{\beta^3}{6} r(r-1)(r-2) &\in \mathbb{Z}. 
    \end{split} \tag{I$\beta$3}
\end{align*}
Now, in the specific case when $\beta=\frac{1}{k}$ with $k\in \mathbb{Z}^+$, we get $r,kc,2k^2d,6k^3e\in \mathbb{Z}$ and our numerical conditions become the integral inequalities:
\begin{equation}
\label{eqnbeta1}
    0<2k^2 d < 2k^2 D \tag{N$\beta$1}
\end{equation}
\begin{equation}
\label{eqnbeta2}
0<kc(6k^3e) < k^4 \min \{ 4d^2 , 4 \left( D-d \right)^2 \} \tag{N$\beta$2}
\end{equation}
\begin{equation}
\label{eqnbeta3}
- \frac{c(2D-2d)}{6e}-R \leq r \leq \frac{c(2d)}{6e}. \tag{N$\beta$3=N3}
\end{equation}
Thus, the correct order for computing the destabilizing Chern characters should be
\[
\begin{tikzcd}
    \eqref{eqnbeta1}\arrow{r} & \eqref{eqnbeta2}\arrow{r} & \eqref{eqnbeta3} \arrow{r} & \eqref{eqibeta1} \arrow{r} & \eqref{eqibeta2} \arrow{r} & \eqref{eqibeta3}
\end{tikzcd}
\]

\section{Bounding the Gieseker wall}\label{sec:Giesekerwall}
In \cite{vertical2023} was found that when looking in the $(\alpha,s)$-slice there was a smallest wall $\alpha_0=1$ with the additional property that the moduli space of Bridgeland semistable objects was empty below it, i.e., for $(6s+1)\alpha^2<1$;  this was done by showing that no values of $\alpha_0<1$ could produce actual walls. In this section, we will give a linear bound in $\ch_2(E)=D$ for the maximum numerical wall, after which the moduli space is isomorphic to the Gieseker moduli space. 

\vspace{0.3cm}
\noindent
Only for this section, we will denote the maximum numerical wall for $v=(-R,0,D,0)$ by $\alpha_{\infty}^D$, considering it as a number in the $\alpha$-axis. 
\begin{proposition}
\label{prop4.1}
    Let $\ch(E)=(-R,0,D,0)$ with $R,D\geq 0$. Then $
    D-2 < \alpha_{\infty}^D \leq D
    $.
\end{proposition}
We will prove this with two lemmas. 
\begin{lemma}\label{lem:Dodd}
    If $D$ is odd, then $D-1<\alpha_{\infty}^D\leq D$.
\end{lemma}
\begin{proof}
    It is enough to show that the maximum wall is produced by one of the following Chern characters: $\ch_0 (A)=0$, $\ch_1 (A)=1$, $\ch_2 (A)=\frac{D}{2}$ and 
    \[
    \ch_3(A)=\frac{D^2-k}{6} \quad \textnormal{ for some } k=0, \,\ldots,\, 5,
    \]
    and having in mind that for $D>3$, $2D-1>5$ and thus $D-1< \sqrt{D^2-5} < \alpha_{\infty}^D$. 

    \vspace{0.3cm}
    \noindent
    If $A\hookrightarrow E$ is a possible destabilizing object with $v=(r,c,d,e)$ then we need to maximize the value $\alpha_0=\frac{6e}{c}$ restricted to the inequalities
    \[
    0\leq c(6e) \leq \min\{4d^2, 4(D-d)^2\}
    \]
    with $c,6e\in \mathbb{Z}$. It can be shown that such a maximum occurs when $6e$ is maximal and $c$ is minimal. Trying to achieve a maximum for the function $f(d)=\min\{4d^2, 4(D-d)^2\}$ with $0<2d<2D$, we note that $f$ is smooth and differentiable everywhere except at its maximum $2d=D$. Thus, the maximum value of $(6e)c$ is achieved when $c=1$, $d=D/2$, and $6e=D^2-k$, with $k$ being the minimal integer value such that $e- \frac{1}{6}\in \mathbb{Z}$, i.e.,
    $$
    k=\begin{cases}
        2 & \text{if}\ D\equiv 3 \mod 6\\
        0 & \text{otherwise.}
    \end{cases}
    $$
    Using the third inequality (\ref{num3}), we get 
    \[
    -\frac{D}{D^2-k} -R \leq r \leq \frac{D}{D^2-k},
    \]
    which means that $r=0$ is always a possibility (and the unique one when $R=0$). Just remains to show that such Chern characters in fact hold the integral conditions. We verify them as follows:
    \begin{itemize}
        \item[I1.] $d-\frac{c^2}{2}= \frac{D-1}{2}\in \mathbb{Z}$ as $D$ is odd.  
        \item[I2.] We chose $6e$ such that $e-\frac{c}{6} \in \mathbb{Z}$.
        \item[I3.] $2e-cd+\frac{c^3}{6} = \frac{2D^2-2k-3D+1}{6}$. However, we note that $3D\equiv 3\mod 6$ and similarly by the second condition $D^2-k-1 \equiv 0 \mod 6$. This implies $2D^2-2k-3D+1\equiv 0 \mod 6$. 
    \end{itemize}
    Thus the Chern characters are possible. 
\end{proof}

\begin{lemma}
    If $D\geq 4$ is even, then $\alpha_{\infty}^D=\alpha_{\infty}^{D-1}$. In particular, $D-2<\alpha_{\infty}^D \leq D-1$. 
\end{lemma}

\begin{proof}
    We know that $f$ reaches its maximum at $D/2$, but for $d=D/2$ condition \eqref{int1} implies that $c$ must be a positive even number, in particular $c\geq 2$ and so
    $$
    \frac{6e}{c}\leq \frac{D^2}{4}.
    $$
    However, taking $c=1$, $2d=D-1$ and $6e=(D-1)^2-k$ for some $k=0,\dots, 5$ for which all the integral conditions are met, we obtain a numerical wall with 
    $$
    \alpha_0^2=(D-1)^2-k\geq D^2/4,\ \ \text{for}\ D\geq 4.
    $$
    Thus, the numerical solutions that produce the maximum wall give exactly $\alpha_{\infty}^{D-1}$.
\end{proof}
\begin{corollary}
\label{corollary4.4}
If $\displaystyle \left(s+\frac{1}{6}\right)\alpha^2>\frac{D^2}{6}$ then every Gieseker semistable sheaf $E$ with $\ch(E)=(0,0,D,0)$ is $\lambda_{0,\alpha,s}$-semistable.
\end{corollary}
\noindent
The bounds above give an idea of the maximum numerical wall that can occur. To obtain a maximum actual wall we should be able to give a short exact sequence in the category, and in general, to be able to find such sequence is hard. We present some extra conditions that a destabilizing subobject must satisfy if it destabilizes a one-dimensional sheaf (i.e., $R=0$). The following result uses some known arguments, but we present it here for completeness.  

\begin{proposition}\label{prop:r0c1}
   If $E\in \Coh^{0}(X)$ is a 1-dimensional Gieseker semistable coherent sheaf with $\ch(E)=(0,0,D,0)$ and $A\hookrightarrow E$ is a subobject in $\mathcal{A}^{0,\alpha}$ with $\ch_0(A)=0$ and $\ch_1(A)=1$ that destabilizes $E$ in the $(\alpha,s)$-plane, then $A$ is a coherent sheaf. In particular, if the maximal numerical wall is actual, then it is produced by a destabilizing rank zero coherent sheaf.
\end{proposition}
\begin{proof}
   The long exact sequence of $\Coh^0(\mathbb{P}^3)$-cohomologies shows that $A\in \Coh^0(\mathbb{P}^3)$. Considering the exact sequence
$$ 
0\rightarrow \mathcal{H}^{-1}(A)[1] \rightarrow  A \rightarrow  \mathcal{H}^{0}(A) \rightarrow 0
$$
    in $\Coh^0(\mathbb{P}^3)$, we get that 
    $$
    0\leq \ch_1(\mathcal{H}^{-1}(A)[1])\leq 1.
    $$
    If $\ch_1(\mathcal{H}^{-1}(A)[1])=1$ then $\ch_1(\mathcal{H}^0(A))=0$. However, $\mathcal{H}^0(A)\in \mathcal{T}_0$ so this is only possible if $\ch_0(\mathcal{H}^0(A))=0$, i.e., if $\mathcal{H}^{-1}(A)=0$ and $A\in\Coh(X)$. On the other hand, consider the destabilizing sequence in $\mathcal{A}^{0,\alpha}$:
    $$
    0\rightarrow A\rightarrow E\rightarrow B\rightarrow 0
    $$
    and assume that $\ch_1(\mathcal{H}^{-1}(A)[1])=0$. Taking the long exact sequence of $\Coh^0(\mathbb{P}^3)$-cohomologies we get two exact sequences in $\Coh^0(\mathbb{P}^3)$:
    \begin{eqnarray}
    & 0\rightarrow B^{-1}\rightarrow A\rightarrow M\rightarrow 0,\label{short1}\\ 
    & 0\rightarrow M\rightarrow E\rightarrow B^0\rightarrow 0,\label{short2}
    \end{eqnarray}
    where $B^{-1}:=\mathcal{H}^{-1}_{\Coh^0}(B)\in \mathcal{F}_{0,\alpha}$ and $B^{0}:=\mathcal{H}^0_{\Coh^0}(B)\in\mathcal{T}_{0,\alpha}$. From the sequence \eqref{short2} we get that $M\in \mathcal{T}_0\subset \Coh(\mathbb{P}^3)$ and since $E$ is 1-dimensional, if $\mathcal{H}^{-1}(B^0)\neq 0$ we would have
    $$
    0\geq \mu_0(\mathcal{H}^{-1}(B^0))=\mu_0(M)>0,
    $$
    which is impossible. Therefore, $B^0\in \mathcal{T}_0$ and $M$ is a 1-dimensional sheaf.
    
    \vspace{0.3cm}
    \noindent
    Now, from the sequence \eqref{short1} we obtain $\mathcal{H}^{-1}(B^{-1})\cong \mathcal{H}^{-1}(A)$, which implies that $\mathcal{H}^{-1}(A)[1]$ is a subobject of $B^{-1}$ in $\Coh^0(\mathbb{P}^3)$, a contradiction if $\mathcal{H}^{-1}(A)\neq 0$ since $\nu_{0,\alpha}(\mathcal{H}^{-1}(A)[1])=+\infty$ and $B^{-1}\in \mathcal{F}_{0,\alpha}$. 

    \vspace{0.3cm}
    \noindent
    For the last part, notice that since in the outermost chamber the only Bridgeland semistable objects of Chern character $(0,0,D,0)$ are the 1-dimensional Gieseker semistable sheaves, then the outermost wall must be produced by a subobject $A\hookrightarrow E$ of a 1-dimensional sheaf $E$ in $\mathcal{A}^{0,\alpha}$. Thus, the result follows from the first part since the maximal numerical wall is given by a Chern character of the form $(0,1,d,e)$ as explained in the proof of Lemma \ref{lem:Dodd}.
\end{proof}
\noindent
We note that Proposition \ref{prop4.1} and Corollary \ref{corollary4.4} refer exclusively to the case $\beta=0$, but we can use a simpler argument to obtain another bound when $\beta\neq 0$. 

\begin{proposition}
\label{prop4.7}
    If $\displaystyle \left(s+\frac{1}{6}\right)\alpha^2>\frac{2D^2}{3}$ then every Gieseker semistable sheaf $E$ with $\ch(E)=(0,0,D,E)$ is $\lambda_{\beta,\alpha,s}$-semistable for $\beta=\frac{E}{D}$.
\end{proposition}

\begin{proof}
    Let us suppose that a wall occurs for an $\alpha_0 > 2D$. Then combining inequalities \eqref{num1}, \eqref{num2} and the definition of $\alpha_0$ we get 
    \[
        4c^2d^2 < 4D^2 c^2 < c(6e) \leq \min \{4d^2, 4(D-d)^2\} \leq 4d^2.
    \]
    Thus $0<c<1$. However, since $\alpha_0^2>2D$ then inequality \eqref{num3} tells us that $r=0$ and so $c\in \mathbb{Z}$, a contradiction.
\end{proof}
\noindent
The last sentence of the proof above gives a point after which every actual wall must come from a destabilizing subobject with rank $0$. More precisely, we have the following:

\begin{corollary}
\label{cor4.8}
    Suppose that $\ch(E)=(0,0,D,E)$ and let $A$ be a destabilizing subobject of $E$ which generates a wall in the region $\displaystyle \left(s+\frac{1}{6}\right)\alpha^2>\frac{D}{3}$ inside the slice $\beta=\frac{E}{D}$, then $\ch_0(A)=0$. 
\end{corollary}
\noindent
We present the phenomena of Corollary \ref{cor4.8} and Proposition \ref{prop4.1} in Figure \ref{fig:5}. 
\begin{figure}
        \centering
        \includegraphics[width=0.5\linewidth]{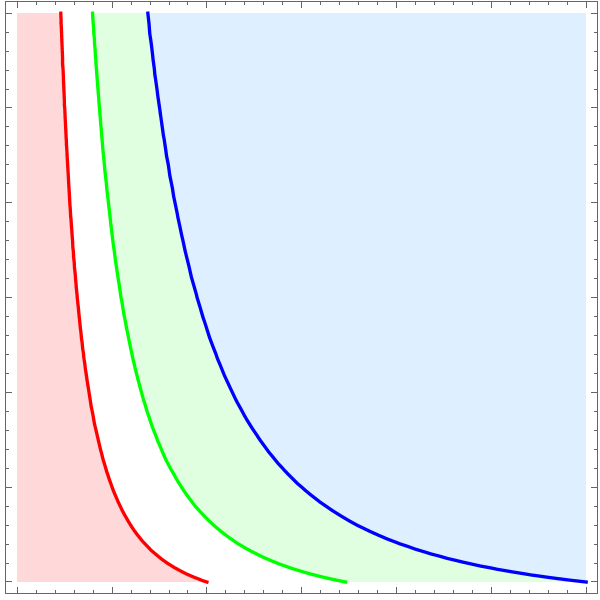}
        \caption{The general figure of walls in $\beta=0$. The red curve represents the killing wall and the blue curve represents the maximal numerical wall, after which the Bridgeland moduli space is the Gieseker moduli space, and in the green region the destabilizing objects have rank $0$.}
        \label{fig:5}
    \end{figure}

\begin{remark}
    In particular, this situation works for any Fano threefold, as the computations are merely numerical. Note that the only condition where we used that we are seated on $\mathbb{P}^3$ was in the integral conditions (\ref{int3} and \ref{eqibeta3}) which come from the Euler characteristic of $\mathbb{P}^3$. Thus, it is also possible to find the maximal numerical wall for those cases.
\end{remark}    
\noindent
We end this Section with a compilation of the results so far.

\renewcommand\proofname{Proof of Theorem \ref{thm1.1}}
\begin{proof}
    The first part is just a restatement of Proposition \ref{prop4.7}. Now, Proposition \ref{prop3.3} assures that for $D=4$ and $\beta=0$ there are three actual walls, and Lemma \ref{lemma3.1} implies their existence for all $D$. The new bound in this case is given by Proposition \ref{prop4.1}. 
\end{proof}


\section{Low degree examples}\label{sec:examples}

\subsection{Degree 3}
It was shown in \cite{vertical2023} that for $D=3$ and $\beta=0$ there are two actual walls: the killing wall $\alpha_0=1$ and $\alpha_0 = \sqrt{7}$. 
In particular, Table \ref{table:1} was obtained, describing all the Chern characters that appeared with its corresponding walls and possible destabilizing sequences.
\begin{table}[h!]
\caption{Numerical walls for $(0,0,3,0)$ and possible destabilizing sequences. From \cite[Example 6.15]{vertical2023}.}\label{table:1}
\begin{tabular}{ccc}
\hline
$\ch(A)$                        & $\alpha_0$ & Possible Destabilizing Sequence \\ \hline
$\left(r,1,\frac{3}{2},\frac{7}{6}\right)$                                &   $\sqrt{7} $        & $$
\begin{tikzcd}[column sep=small,ampersand replacement=\&]
    0 \arrow{r} \& \mathcal{O}_H(2) \arrow{r} \& \mathcal{O}_C(2) \arrow{r} \& \mathcal{O}_H(-1)[1] \arrow{r} \& 0
\end{tikzcd}
$$ with $\deg C = 3$                     \\[0.2cm]
$\left(r,1,\frac{1}{2},\frac{1}{6}\right)$& 1 &  $$
\begin{tikzcd}[column sep=small,ampersand replacement=\&]
    0 \arrow{r} \& \mathcal{O} (1) \arrow{r} \& E \arrow{r}  \& \mathcal{R}[1] \arrow{r} \& 0 
\end{tikzcd}
$$ \\[0.2cm]
$\left(r,2,1,\frac{1}{3}\right)$                                &   1         &     $$
\begin{tikzcd}[column sep=small, ampersand replacement=\&]
    0 \arrow{r} \& \mathcal{O}(1)^{\oplus 2} \arrow{r} \& E \arrow{r}  \& \mathcal{S}[1] \arrow{r} \& 0
\end{tikzcd}
$$  \\
$\left(r,3,\frac{3}{2},\frac{1}{2}\right)$ & 1          & $$
\begin{tikzcd}[column sep=small, ampersand replacement=\&]
0 \arrow{r} \& \mathcal{O} (1)^{\oplus 3} \arrow{r} \& E \arrow{r}  \& \mathcal{K}[1] \arrow{r} \& 0
\end{tikzcd}
$$\\[0.2cm]
$\left(r,1,\frac{5}{2},\frac{1}{6}\right)$& 1           & $$
\begin{tikzcd}[column sep=small, ampersand replacement=\&]
0 \arrow{r} \& \mathcal{R}^{\vee}[1]  \arrow{r} \& E^{\vee}[2] \arrow{r} \& \mathcal{O}(-1)[2] \arrow{r} \& 0
\end{tikzcd}
$$\\[0.2cm]
$\left(r,2,2,\frac{1}{3}\right)$ &    1        & $$
\begin{tikzcd}[column sep=small, ampersand replacement=\&]
0 \arrow{r} \& \mathcal{S}^{\vee}_4[1]  \arrow{r} \& E^{\vee}[2] \arrow{r} \& \mathcal{O}(-1)^{\oplus 2}[2] \arrow{r} \& 0\end{tikzcd}
$$\\[0.2cm]
$\left(r,1,\frac{3}{2},\frac{1}{6}\right)$ &   1   & $$
\begin{tikzcd}[column sep=small,ampersand replacement=\&]
    0 \arrow{r} \& \mathcal{I}_{p,H} (2) \arrow{r} \& E \arrow{r}  \& 
    U \arrow{r} \& 0 
\end{tikzcd}
$$                        \\ \hline
\end{tabular}
\end{table}

\vspace{0.2cm}
\noindent
Now, when $\beta=\frac{1}{3}$ and $\ch^{\beta}=(0,0,3,0)$, we get $\ch=(0,0,3,1)$ and thus the Bridgeland moduli space in the outermost chamber is no other than the moduli space $\mathcal{M}_{3t+7}$. By \cite[Theorem 1.1]{cubics2004}, we know that any sheaf $\mathcal{F}$ in $\mathcal{M}_{3t+7}$ is equivalent to a complex
\[
\mathcal{F}\stackrel{\text{qis}}{\cong}
\left[
    \mathcal{O}_{\mathbb{P}^3}(-1)^{\oplus 2} \rightarrow \mathcal{O}_{\mathbb{P}^3}(1) \oplus \mathcal{O}_{\mathbb{P}^3}^{\oplus 3} \rightarrow \mathcal{O}_{\mathbb{P}^3}(2) \oplus \mathcal{O}_{\mathbb{P}^3}(1)
\right],
\]
and therefore, has $\mathcal{O}_{\mathbb{P}^3} (1)$ as a subobject in the category $\mathcal{A}^{\beta,\alpha}$. Thus, we have the following:
\begin{lemma}
\label{lemma5.2}
    Let $\beta=\frac{1}{3}$ and $\ch^\beta = (0,0,3,0)$. Any wall in the $(\alpha,s)$-slice destabilizing a Gieseker semistable sheaf must lie in the region
    \[
    \bigg\{
    (\alpha,s) \, \bigg| \, \left( s+\frac{1}{6} \right)\alpha^2 > \frac{2}{27}
    \bigg\}.
    \]
    In particular, $\alpha_0 \geq 1-\beta = \frac{2}{3}$. 
\end{lemma} 
Before describing the possible walls, it is convenient to identify a couple of familiar objects in the category $\mathcal{A}^{\beta,\alpha}$:
\begin{lemma}\label{lemma:familiaarobjects}
    Let $Z\subset H\subset \mathbb{P}^3$ be a zero-dimensional subscheme supported on a plane $H$. Then
    \begin{enumerate}
        \item $\mathcal{O}_H(-k),\mathcal{O}_Z\in\mathcal{T}_{\beta}$ for any $\beta$;
        \item $\mathcal{O}_Z$ is $\nu_{\beta,\alpha}$-semistable for any $\beta$ and $\alpha>0$. Moreover, $\mathcal{O}_Z\in\mathcal{T}_{\beta,\alpha}\subset \mathcal{A}^{\beta,\alpha}$;
        \item $\mathcal{O}_H(-k)$ is $\nu_{\beta,\alpha}$-semistable for any $\alpha>0$ as long as $\beta\geq -k$ or $\beta\leq -k-1$;
        \item $\mathcal{O}_H(-k)\in \mathcal{F}_{\beta,\alpha}$ for any $\alpha>0$ as long as $\beta\geq -k$. In particular, $\mathcal{O}_H(-k)[1]\in \mathcal{A}^{\beta,\alpha}$ for $\beta\geq -k$;
        \item If $\mathcal{I}_{Z,H}$ denotes the ideal of $Z$ on the plane $H$, then the object ${E_{Z,k}}=R\mathcal{H}om(\mathcal{I}_{Z,H}(k),\mathcal{O}_H)[1]\in \mathcal{A}^{\beta,\alpha}$ for $\beta\geq -k$ and any $\alpha>0$.
    \end{enumerate}
\end{lemma}
\renewcommand{\proofname}{Proof}
\begin{proof}
    Parts (1) and (2) follow easily since $\mathcal{O}_H(-k)$ and $\mathcal{O}_Z$ are torsion sheaves and $\nu_{\beta,\alpha}(\mathcal{O}_Z)=+\infty$. For part (3), notice that the tilt-walls in the $(\beta,\alpha)$-plane for the Chern character $\ch(\mathcal{O}_H(-k))$ are semicircles with center $\left(-k-\frac{1}{2},0\right)$ and with radii bounded by
    $$
    \frac{\sqrt{\Delta(\mathcal{O}_H(-k))}}{2}=\frac{1}{2}
    $$
    (see \cite[Lemma 7.2]{macrischmidt19}). Thus,  $\mathcal{O}_H(-k)$ is $\nu_{\beta,\alpha}$-semistable for $\beta\geq -k$ or $\beta\leq -k-1$ since it is a 2-Gieseker semistable sheaf. For part (4) we only need to notice that
    $$
    \nu_{\beta,\alpha}(\mathcal{O}_H(-k))=-\left(k+\beta+\frac{1}{2}\right),
    $$
    and in particular $\nu_{\beta,\alpha}(\mathcal{O}_H(-k))<0$ for $\beta\geq k$. For part (5), first notice that
    $$
    \mathcal{H}^j(E_{Z,k})=\begin{cases}
        \mathcal{O}_H(-k) & if\ j=-1\\
        \mathcal{O}_Z & if\ j=0\\
        0 & \text{otherwise}.
    \end{cases}
    $$
    Thus, $\mathcal{H}^{-1}_{\Coh^{\beta}}({E_{Z,k}})=\mathcal{H}^{-1}(E_{Z,k})\in \mathcal{T}_{\beta}$ and $\mathcal{H}^{0}_{\Coh^{\beta}}({E_{Z,k}})=\mathcal{H}^{0}(E_{Z,k})\in \mathcal{T}_{\beta}$. Moreover, by (4) we obtain that for $\beta\geq -k$ the complex $E_{Z,k}$ is the object in $\mathcal{A}^{\beta,\alpha}$ given by the extension
    $$
    0\rightarrow \mathcal{O}_H(-k)[1]\rightarrow E_{Z,k}\rightarrow \mathcal{O}_Z\rightarrow 0.
    $$
\end{proof}
\begin{proposition}
\label{prop5.3}
    Let $\ch^{\beta}=(0,0,3,0)$. If $\beta=\frac{1}{3}$, the boundary of the Gieseker chamber is the wall $\alpha_0=\sqrt{\dfrac{13}{3}}$, given by the destabilizing subobject $\mathcal{O}_H(2)$. 
\end{proposition}

\begin{proof}
    After tuning the algorithm (see Appendix \ref{code}) with the condition of Lemma \ref{lemma5.2}, we obtain the ($\beta$-twisted) Chern characters of possible destabilizing subobjects: 
    $$
    \left(0,1,\frac{7}{6}, \frac{13}{18}\right),\ \left(0,1,\frac{13}{6}, \frac{7}{18}\right).
    $$
    The largest of these numerical walls is $\alpha_0^2=\frac{6e}{c}=\frac{13}{3}$ and is given by the (untwisted) Chern character $\ch=(0,1,\frac{3}{2},\frac{7}{6})$. If this wall is actual, this would be the boundary of the Gieseker chamber. If this wall is actual, a similar argument to the one used in the proof of Proposition \ref{prop:r0c1} implies that the destabilizing subobject must be a sheaf supported on a plane, then we can use Remark \ref{p2p3correspondence} to guess the possible destabilizing sequence. In fact, as explained in \cite{cubics2004}, there is a whole component of the moduli space $\mathcal{M}_{3t+7}$ consisting of 1-dimensional sheaves $\mathcal{F}$ fitting into exact sequences 
    \begin{equation}\label{longexactsequenceexample31}
    0\rightarrow \mathcal{O}_H(-1)\rightarrow \mathcal{O}_H(2)\rightarrow \mathcal{F}\rightarrow \mathbb{C}_p\rightarrow 0.
    \end{equation}
    These sheaves fit into the exact triangle:
    \begin{equation}\label{destabilizing31}
    \mathcal{O}_H(2)\rightarrow \mathcal{F}\rightarrow E_{p,1}\rightarrow \mathcal{O}_H(2)[1],
    \end{equation}
    where $E_{p,1}$ is as defined in Lemma \ref{lemma:familiaarobjects}. Moreover, since $1/3\geq -1$, then Lemma \ref{lemma:familiaarobjects} implies that $E_{p,k}\in \mathcal{A}^{1/3,\alpha}$ for all $\alpha>0$ and therefore the triangle \eqref{destabilizing31} gives us the short exact sequence
    $$
    0\rightarrow \mathcal{O}_H(2)\rightarrow \mathcal{F}\rightarrow E^{\bullet}\rightarrow 0
    $$
    in $\mathcal{A}_{1/3,\alpha}$. Since $\ch(\mathcal{O}_H(2))=(0,1,\frac{3}{2},\frac{7}{6})$, then the outermost numerical wall is actual, the boundary of the Gieseker chamber, and destabilizes precisely the sheaves given by \eqref{longexactsequenceexample31}.
\end{proof}
\begin{remark}
    The Gieseker moduli $\mathcal{M}_{3t+7}$ has two irreducible components, one consisting of the plane sheaves destabilized at the Gieseker wall and another consisting of twisted cubic curves, the intersection of the two components parametrizes singular plane cubics with an embedded point at the singularity (see \cite{cubics2004}). After crossing the Gieseker wall in Proposition \ref{prop5.3} no new stable objects are created since
    $$
    \Ext^1(\mathcal{O}_H(2),E^{\bullet})\cong \Hom(\mathcal{O}_H,\mathcal{I}_{p,H}(1)\otimes\mathcal{O}_H(2-3))^{\vee}=H^0(\mathbb{P}^2;\mathcal{I}_p)=0.
    $$
    However, the moduli space of Bridgeland semistable objects in the chamber adjacent to the Gieseker chamber is a proper algebraic space so there should be on this chamber semistable objects that are not sheaves. The authors do not know how these objects look like, but the existence of semistable objects that do not come from destabilized sheaves is a known phenomena (see for instance \cite[Example 6.6]{vertical2023}).
\end{remark}
\begin{remark}
    It is known that the moduli space $\mathcal{M}_{3t+2}$ is dual to $\mathcal{M}_{3t-2}\cong\mathcal{M}_{3t+7}$ and thus we can obtain the boundary for the Gieseker chamber in this case by dualizing (instead of looking at the case $\beta= \frac{2}{3}$). 
\end{remark}
\begin{remark}
One could ask what happens in the complementary region from Lemma \ref{lemma5.2}, this is, the region where walls (if any) destabilize stable complexes. In general, the number of numerical walls increases rapidly as we get near $\{s=0\}$, the Chern characters become unfamiliar as well, and our algorithm has difficulties computing them fast (note that the conditions when $k$ grows allow a huge amount of possibilities that should be checked).
\end{remark}
\subsection{Degree 4}
Using a process similar to the one above, we will also describe all numerical walls for the Chern character $\ch^\beta = (0,0,4,0)$ in the $(\alpha,s)$-slice. For $\beta = 0$ we get 
\begin{proposition}
\label{prop6.1}
    Let $\ch = (0,0,4,0)$. Then there are three actual walls $\alpha_0^2= 1,4,$ and $7$. 
\end{proposition}

\begin{proof}
    We use the algorithm to compute all numerical walls (see Table \ref{table:2}). The values correspond exactly to the planar, the elliptical and the killing wall as discussed in Proposition \ref{prop3.3}. 
\end{proof}

\begin{lemma}\label{regionforbeta1:4}
    Let $\beta=\frac{1}{4}$ and $\ch^\beta = (0,0,4,0)$. Any wall in the $(\alpha,s)$-slice where a Gieseker semistable sheaf is destabilized must lie in the region
    \[
    \bigg\{
    (\alpha,s) \, \bigg| \, \left( s+\frac{1}{6} \right)\alpha^2 > \frac{3}{32}
    \bigg\}
    \]
    In particular, $\alpha_0 \geq 1-\beta = \frac{3}{4}$. 
\end{lemma}

\begin{proof}
    As for the case $\ch^{1/3}=(0,0,3,0)$, from \cite[Theorem 1.2 \& Theorem 6.1]{quartics2016} we get that every sheaf in any of the components of $\mathcal{M}_{4t+9}$ has $\mathcal{O}(1)$ as a subobject in $\mathcal{A}^{\beta,\alpha}$. 
\end{proof}
\noindent
Now, using our algorithm we can compute those walls. We present the outcomes in Table \ref{table:3}. Moreover, we obtain the following:

\begin{proposition}
\label{prop6.3}
    Let $\ch^\beta = (0,0,4,0)$. If $\beta=\frac{1}{4}$, then there are four pseudo-walls destabilizing Gieseker semistable sheaves, with at least the first two of them (in decreasing order of $\alpha_0$) being actual:
    \begin{enumerate}[label=\roman*.]
        \item\label{prop6.3-i}$\alpha_0=\dfrac{\sqrt{151}}{4}$, produced by the subobject $\mathcal{I}_{p,H}(3)$,
        \item\label{prop6.3-ii}$\alpha_0=\dfrac{\sqrt{79}}{4}$, produced by the subobject $\mathcal{O}_H(2)$.
    \end{enumerate}
    The other two numerical walls with $\alpha_0 \geq \frac{3}{4}$ are $\alpha_0 = \, \frac{\sqrt{55}}{4}$ and $\frac{\sqrt{31}}{4}$.
\end{proposition}

\begin{proof}
    The four numerical walls are obtained using our algorithm tuned to the condition from Lemma \ref{regionforbeta1:4} (see Table \ref{table:3}). The wall \ref{prop6.3-i} comes from the twisted Chern character $\ch^\beta=\left(0,1,\dfrac{9}{4}, \dfrac{151}{96}\right)$ corresponding to $\ch=\left(0,1,\dfrac{5}{2}, \dfrac{13}{6}\right)$, while the wall \ref{prop6.3-ii} comes from the twisted Chern character $\ch^\beta = \left(0,1,\dfrac{5}{4}, \dfrac{79}{96}\right)$ corresponding to $\ch = \left(0,1,\dfrac{3}{2}, \dfrac{7}{6}\right) $. The elements in $\mathcal{M}_{4t+9}$ have a component of planar sheaves, so we might try to use Remark \ref{p2p3correspondence} again to induce Bridgeland walls from $\mathbb{P}^2$. In \cite{drezetmaican11}, it was shown that a one-dimensional Gieseker semistable plane sheaf with Hilbert polynomial $4t+9$ fits into one of the following exact sequences:
    \begin{eqnarray}
        0 \rightarrow\mathcal{O}_H (-1) \oplus \mathcal{O}_H (1) \rightarrow \mathcal{O}_H (2)^{\oplus 2} \rightarrow \mathcal{F} \rightarrow 0,\label{deg4closedloci} \\
        0\rightarrow (\mathcal{O}_H)^{\oplus 3}\rightarrow \mathcal{O}_H(1)^{\oplus 2}\oplus \mathcal{O}_H(2)\rightarrow \mathcal{F}\rightarrow 0.\label{deg4openloci}
    \end{eqnarray}
    The sheaves from \eqref{deg4closedloci} fit into triangles
    $$
    \mathcal{I}_{p,H}(3)\rightarrow \mathcal{F}\rightarrow \mathcal{O}_H(-1)[1]\rightarrow \mathcal{I}_{p,H}(3)[1], 
    $$
    while the sheaves from \eqref{deg4openloci} fit into triangles
    $$
    \mathcal{O}_H(2)\rightarrow \mathcal{F}\rightarrow E_{Z,2}\rightarrow \mathcal{O}_H(2)[1], 
    $$
    where $Z\subset H\subset \mathbb{P}^3$ is a 0-dimensional subscheme of length 3. As before, Lemma \ref{lemma:familiaarobjects} shows that these triangles produce short exact sequences in $\mathcal{A}^{\frac{1}{4},\alpha}$:
    \begin{eqnarray}
         0\rightarrow \mathcal{I}_{p,H}(3)\rightarrow \mathcal{F}\rightarrow \mathcal{O}_H(-1)[1]\rightarrow 0\\
    0\rightarrow \mathcal{O}_H(2)\rightarrow \mathcal{F}\rightarrow E_{Z,2}\rightarrow 0,
    \end{eqnarray}
    which are the destabilizing sequences responsible for the walls \ref{prop6.3-i} and \ref{prop6.3-ii}, respectively.
\end{proof}
\noindent
Modulo dualizing and twisting there are only three moduli spaces of 1-dimensional Gieseker semistable space sheaves of degree 4, namely $\mathcal{M}_{4t},\ \mathcal{M}_{4t+1},\  \mathcal{M}_{4t+2}$. The boundary of the Gieseker chambers for the first two cases were computed in Proposition \ref{prop6.1} and Proposition \ref{prop6.3}, respectively. For the remaining case we have the following:
\begin{proposition}
\label{prop6.4}
    Let $\ch^\beta = (0,0,4,0)$. If $\beta=\frac{1}{2}$, then there are seven numerical walls destabilizing Gieseker semistable sheaves, with at least the first three of them (in decreasing order of $\alpha_0$) being actual:
    \begin{enumerate}[label=\roman*.]
        \item\label{prop6.4-i}$\alpha_0=\dfrac{7}{2}$, produced by the subobject $\mathcal{O}_H(3)$,
        \vspace{0.1cm}
        \item\label{prop6.4-ii} $\alpha_0=\dfrac{5}{2}$, produced by subobjects $\mathcal{I}_{p,H} (3)$, and
        \vspace{0.1cm}
        \item\label{prop6.4-iii} $\alpha_0 = \dfrac{\sqrt{13}}{2}$, produced by the subobject $\mathcal{O}_H (2)$. 
    \end{enumerate}
    The four other numerical walls $\alpha_0 \geq \frac{1}{2}$ are $\alpha_0 = \, 1, \frac{1}{\sqrt{2}}, \, \sqrt{\frac{5}{14}}$ and $\frac{1}{2}$ . 
\end{proposition}

\begin{proof}
    The full computations of the pseudo-walls obtained by the algorithm can be found in Table \ref{table4}. The first three walls are actual because again they come from the walls destabilizing plane sheaves with respect to tilt stability. These walls were studied in detail in \cite[Section 6.1]{BMW} and the corresponding destabilizing sequences are
    \begin{eqnarray}
        0\rightarrow \mathcal{O}_H(3)\rightarrow \mathcal{F}\rightarrow \mathcal{O}_H(-1)[1]\rightarrow 0\\
        0\rightarrow \mathcal{I}_{p,H}(3)\rightarrow \mathcal{F}\rightarrow E_{q,1}\rightarrow 0\\
        0\rightarrow \mathcal{O}_H(2)^{\oplus 2}\rightarrow \mathcal{F}\rightarrow {\mathcal{O}_H}^{\oplus 2}[1]\rightarrow 0,
    \end{eqnarray}
    which are all exact in $\mathcal{A}^{1/2,\alpha}$ by Lemma \ref{lemma:familiaarobjects}.
\end{proof}

\begin{remark}
    Note that in all the examples in this section, the first sheaves that are destabilized are plane sheaves and in particular the boundary of the Gieseker chamber coincides with the tilt wall corresponding to the boundary of the nef cone of the moduli space $\mathcal{M}_{p(t)}(\mathbb{P}^2)$ as computed in \cite{Woolf}. 
\end{remark}

\renewcommand\proofname{Proof of Theorem \ref{thm1.2}}
\begin{proof}
This comes from the adjunction of Lemma \ref{lemma5.2} and Proposition \ref{prop5.3} for $D=3$, and Proposition \ref{prop6.1}, Proposition \ref{prop6.3} and Proposition \ref{prop6.4} for $D=4$.

\end{proof}


\newpage
\appendix
\section{Code}\label{code}
In this \href{https://github.com/daniel-bernalm/walls_computations}{link}, we present the code used to find the numerical walls that appear in Section 4. If $A$ is a destabilizing object of $E$ with Chern character $\ch^{\beta} = (-R, 0 ,D, 0)$ with $\beta= \frac{1}{k}$ or $0$ and $\ch^\beta A = (r,c,d,e)$, then the algorithm computes all possibilities of such Chern character using the integral conditions \ref{eqnbeta1}, \ref{eqnbeta2} and \ref{eqnbeta3} and the numerical conditions \ref{eqibeta1}, \ref{eqibeta2} and \ref{eqibeta3}. In order to improve the speed for computing the cases $D=3$ with $\beta=\frac{1}{3}$ and $D=4$ with $\beta = \frac{1}{4}$, the algorithm relies on the multiprocessing package from Python, and thus the program would work better on a physical computer with at least 4 cores. 

\vspace{0.3cm}
\noindent
In this case, we define an object by calling its class \begin{verbatim}
    a=Sheaf(R,D,k)
\end{verbatim} where $\beta=\frac{1}{k}$. If $k=1$ then the program assumes $\beta=0$. After it you use the method 
\begin{verbatim}
    a.num_dest(d)
\end{verbatim}
where the argument counts the number for which you want the process to be divided. This is used for the multiprocessing aspect, each process analyzes simultaneously $\frac{1}{d}$ of the list. The output is then a list of numbers of the form $\left([r_i]_{i=1}^l,c,d,e\right)$ where the $r_i$ are the possible ranks, and $c,d,e$ are the Chern characters of the destabilizing object generating the pseudo-wall. 

\vspace{0.3cm}
\noindent
For example, in the known case $R=0$, $D=3$ and $\beta=0$ studied in \cite[Example 6.15]{vertical2023} the program computes:
\begin{lstlisting}[breaklines]
    [([-5, -4, -3, -2, -1, 0, 1], Fraction(1, 1), Fraction(1, 2), Fraction(1, 6)), ([-4, -3, -2, -1, 0, 1, 2], Fraction(2, 1), Fraction(1, 1), Fraction(1, 3)), ([-3, -2, -1, 0, 1, 2, 3], Fraction(1, 1), Fraction(3, 2), Fraction(1, 6)), ([0], Fraction(1, 1), Fraction(3, 2), Fraction(7, 6)), ([-3, -2, -1, 0, 1, 2, 3], Fraction(3, 1), Fraction(3, 2), Fraction(1, 2)), ([-21, -20, -19, -18, -17, -16, -15, -14, -13, -12, -11, -10, -9, -8, -7, -6, -5, -4, -3, -2, -1, 0, 1, 2, 3, 4, 5, 6, 7, 8, 9, 10, 11, 12, 13, 14, 15, 16, 17, 18, 19, 20, 21], Fraction(7, 1), Fraction(3, 2), Fraction(1, 6)), ([-2, -1, 0, 1, 2, 3, 4], Fraction(2, 1), Fraction(2, 1), Fraction(1, 3)), ([-1, 0, 1, 2, 3, 4, 5], Fraction(1, 1), Fraction(5, 2), Fraction(1, 6))]
\end{lstlisting}
Coinciding with the results obtained before. 

\newpage
\section{Numerical computations}\label{appendixB}
Here we present all the results obtained by the algorithm for the cases $D=4$ and $\beta=0, \, \frac{1}{4}$ and $\frac{1}{2}$. 

\begin{table}[h]
\caption{Numerical walls for sheaves with $\ch^\beta=(0,0,4,0)$ and $\beta=\frac{1}{4}$.}\label{table:3}
\begin{tabular}{ccc}
\hline
$\ch^\beta(A)$                        & $\alpha_0$ & $\ch(A)$  \\[0.1cm] \hline \\[-0.4cm]
$\left(0,1,\dfrac{9}{4},\dfrac{151}{96}\right)$ & $\dfrac{\sqrt{151}}{4} \cong 3.0721 $ & $\left(0,1,\dfrac{5}{2},\dfrac{13}{6}\right)$           \\[0.3cm]
$\left(0,1,\dfrac{5}{4},\dfrac{79}{96}\right)$ & $\dfrac{\sqrt{79}}{4}\cong 2.2220 $ &  $\left(0,1,\dfrac{3}{2},\dfrac{7}{6}\right)$          \\[0.3cm]
$\left(0,1,\dfrac{9}{4},\dfrac{55}{96}\right)$ & $\dfrac{\sqrt{55}}{4} \cong 1.8540 $  &  $\left(0,1,\dfrac{5}{2},\dfrac{7}{6}\right)$          \\[0.3cm]
$\left(0,2,\dfrac{5}{2},\dfrac{31}{48}\right)$ & $\dfrac{\sqrt{31}}{4} \cong 1.3919 $ &  $\left(0,2,3,\dfrac{4}{3}\right)$          \\[0.3cm]
$\left(0,1,\dfrac{13}{4},\dfrac{31}{96}\right)$ & $\dfrac{\sqrt{31}}{4} \cong 1.3919 $ &  $\left(0,1,\dfrac{7}{2},\dfrac{7}{6}\right)$        \\[0.3cm] \hline
\end{tabular}
\end{table}

\newpage

\begin{table}[h]
\caption{Numerical walls for $(0,0,4,0)$ with $\beta = 0$ and possible destabilizing sequences}\label{table:2}
\begin{tabular}{ccc}
\hline
$\ch(A)$                        & $\alpha_0$ & Possible Destabilizing Sequence \\ \hline
$\left(r,1,\dfrac{3}{2},\dfrac{7}{6}\right)$& $\sqrt{7}$ &  $$
\begin{tikzcd}[column sep=small,ampersand replacement=\&]
    0 \arrow{r} \& \mathcal{O}_H (2) \arrow{r}{\widehat{\sigma}} \& j_{*}L \arrow{r}  \& 
 A:= C(\widehat{\sigma}) \arrow{r} \& 0 
\end{tikzcd}
$$ \\[0.2cm]
$\left(r,1,\dfrac{5}{2},\dfrac{7}{6}\right)$                                &   $\sqrt{7}$         &     $$
\begin{tikzcd}[column sep=small,ampersand replacement=\&]
    0 \arrow{r} \& A^{\vee}[2] \arrow{r} \& j_{*}L^{\vee}[2] \arrow{r}  \& \mathcal{O}_H(-2)[2] \arrow{r} \& 0 
\end{tikzcd}
$$ if $E\in\Ext^1 (\mathcal{O}_H(2),A)\neq 0$   \\[0.2cm]
$\left(r,2,2,\dfrac{4}{3}\right)$                                &   2         & $$
\begin{tikzcd}[column sep=small,ampersand replacement=\&]
    0 \arrow{r} \& \mathcal{O}(2) \arrow{r}{\sigma} \& \mathcal{O}_E (2) \arrow{r} \& I_E(2) [1] \arrow{r} \& 0
\end{tikzcd}
$$                     \\[0.2cm]
$\left(r,1,\dfrac{1}{2},\dfrac{1}{6}\right)$ & 1          & $$
\begin{tikzcd}[column sep=small]
0 \arrow{r} & \mathcal{O}(1) \arrow{r} & E \arrow{r} & \mathcal{R}_1[1] \arrow{r} & 0
\end{tikzcd}
$$\\[0.2cm]
$\left(r,2,1,\dfrac{1}{3}\right)$& 1           & $$
\begin{tikzcd}[column sep=small]
0 \arrow{r} & 2\cdot\mathcal{O}(1) \arrow{r} & E \arrow{r} & \mathcal{R}_2[1] \arrow{r} & 0
\end{tikzcd}
$$\\[0.2cm]
$\left(r,3,\dfrac{3}{2},\dfrac{1}{2}\right)$ &    1        & $$
\begin{tikzcd}[column sep=small]
0 \arrow{r} & 3\cdot\mathcal{O}(1) \arrow{r} & E \arrow{r} & \mathcal{R}_3[1] \arrow{r} & 0
\end{tikzcd}
$$\\[0.2cm]
$\left(r,1,\dfrac{3}{2},\dfrac{1}{6}\right)$ &   1   & $$
\begin{tikzcd}[column sep=small,ampersand replacement=\&]
    0 \arrow{r} \& \mathcal{I}_{p,H} (2) \arrow{r} \& E \arrow{r}  \& 
    U \arrow{r} \& 0 
\end{tikzcd}
$$ if $E\in\Ext^1 (\mathcal{O}_H(2),A)\neq 0$ \\[0.2cm]
$\left(r,2,2,\dfrac{1}{3}\right)$ & 1           &  $$
\begin{tikzcd}[column sep=small,ampersand replacement=\&]
    0 \arrow{r} \& \mathcal{I}_{p} (2) \arrow{r} \& E_1 \arrow{r}  \& 
    U_1 \arrow{r} \& 0 
\end{tikzcd}
$$ if $E_1\in\Ext^1 (\mathcal{O}_p,I_E (2) [1])\neq 0$                           \\
$\left(r,4,2,\dfrac{2}{3}\right)$                                & 1           & \makecell{ $$
\begin{tikzcd}[column sep=small,ampersand replacement=\&]
0 \arrow{r} \& 4\cdot\mathcal{O}(1) \arrow{r} \& E \arrow{r} \& \mathcal{R}_4[1] \arrow{r} \& 0
\end{tikzcd}
$$ \\ $$
\begin{tikzcd}[column sep=small,ampersand replacement=\&]
0 \arrow{r} \& \mathcal{R}^{\vee}_4[1]  \arrow{r} \& E^{\vee}[2] \arrow{r} \& 4\cdot\mathcal{O}(-1)[2] \arrow{r} \& 0
\end{tikzcd}
$$
}   \\[0.4cm]
$\left(r,1,\dfrac{5}{2},\dfrac{1}{6}\right)$                               &  1          & $$
\begin{tikzcd}[column sep=small,ampersand replacement=\&]
    0 \arrow{r} \& U^{\vee}[2]  \arrow{r} \& E^{\vee}[2] \arrow{r}  \& \mathcal{O}_{p,H}(-2)[1] \& 0 
\end{tikzcd}
$$                            \\
$\left(r,3,\dfrac{5}{2},\dfrac{1}{2}\right)$                                &  1          &   $$
\begin{tikzcd}[column sep=small,ampersand replacement=\&]
0 \arrow{r} \& \mathcal{R}^{\vee}_3[1]  \arrow{r} \& E^{\vee}[2] \arrow{r} \& 3\cdot\mathcal{O}(-1)[2] \arrow{r} \& 0
\end{tikzcd}
$$                          \\
$\left(r,2,3,\dfrac{1}{3}\right)$                                &   1         &   $$
\begin{tikzcd}[column sep=small,ampersand replacement=\&]
0 \arrow{r} \& \mathcal{R}^{\vee}_2[1]  \arrow{r} \& E^{\vee}[2] \arrow{r} \& 2\cdot\mathcal{O}(-1)[2] \arrow{r} \& 0
\end{tikzcd}
$$                          \\
$\left(r,1,\dfrac{7}{2},\dfrac{1}{6}\right)$                                &   1         &   $$
\begin{tikzcd}[column sep=small,ampersand replacement=\&]
0 \arrow{r} \& \mathcal{R}^{\vee}_1[1]  \arrow{r} \& E^{\vee}[2] \arrow{r} \& \mathcal{O}(-1)[2] \arrow{r} \& 0
\end{tikzcd}
$$                          \\ \hline
\end{tabular}
\end{table}

\newpage

\begin{longtable}[c]{ W{c}{.30\textwidth} W{c}{.15\textwidth}}
\caption{Numerical walls $\alpha_0\geq\frac{1}{2}$ for $\ch^\beta=(0,0,4,0)$ with $\beta=\frac{1}{2}$.}\\
\hline
$\ch^\beta(A)$ & $\alpha_0^2$  \\ \hline
$\left( \left[0\right], 1 , 2, \dfrac{49}{24} \right)$ & $\dfrac{49}{4}$  \\[0.3cm]
$\left( \left[0\right], 1 , 2, \dfrac{25}{24} \right)$ & $\dfrac{25}{4}$  \\[0.3cm]
$\left( \left[0\right], 1 , \dfrac{1}{1}, \dfrac{13}{24} \right)$ & $\dfrac{13}{4}$  \\[0.3cm]
$\left( \left[0\right], 2 , 2, \dfrac{13}{12} \right)$ & $\dfrac{13}{4}$  \\[0.3cm]
$\left( \left[0\right], 1 , 3, \dfrac{13}{24} \right)$ & $\dfrac{13}{4}$  \\[0.3cm]
$\left( \left[2\right], \dfrac{5}{2} , \dfrac{13}{8}, \dfrac{5}{12} \right)$ & $1$  \\[0.3cm]
$\left( \left[-10\right], \dfrac{3}{2} , \dfrac{5}{8}, \dfrac{1}{8} \right)$ & $\dfrac{1}{2}$ \\[0.3cm]
$\left( \left[-2\right], \dfrac{7}{2} , \dfrac{13}{8}, \dfrac{5}{24} \right)$ & $\dfrac{5}{14}$  \\[0.3cm]
$\left( \left[-28,-12,4\right], 1 , \dfrac{1}{2}, \dfrac{1}{24} \right)$ & $\dfrac{1}{4}$  \\[0.3cm]
$\left( \left[-28,-12,4\right], 2 , \dfrac{1}{2}, \dfrac{1}{12} \right)$ & $\dfrac{1}{4}$  \\[0.3cm]
$\left( \left[-26,-10,6\right], 2 , \dfrac{3}{4}, \dfrac{1}{12} \right)$ & $\dfrac{1}{4}$  \\[0.3cm]
$\left( \left[-24,-8,8\right], 1 , 1, \dfrac{1}{24} \right)$ & $\dfrac{1}{4}$  \\[0.3cm]
$\left( \left[-24,-8,8\right], 2 , 1, \dfrac{1}{12} \right)$ & $\dfrac{1}{4}$   \\[0.3cm]
$\left( \left[-24,-8,8\right], 3 , 1, \dfrac{1}{8} \right)$ & $\dfrac{1}{4}$  \\[0.3cm]
$\left( \left[-24,-8,8\right], 4 , 1, \dfrac{1}{6} \right)$ & $\dfrac{1}{4}$  \\[0.3cm]
$\left( \left[-22,-6,10\right], 2 , \dfrac{5}{4}, \dfrac{1}{12} \right)$ & $\dfrac{1}{4}$  \\[0.3cm]
$\left( \left[-22,-6,10\right], 4 , \dfrac{5}{4}, \dfrac{1}{6} \right)$ & $\dfrac{1}{4}$  \\[0.3cm]
$\left( \left[-20,-4,12\right], 1 , \dfrac{3}{2}, \dfrac{1}{24} \right)$ & $\dfrac{1}{4}$ \\[0.3cm]
$\left( \left[-20,-4,12\right], 2 , \dfrac{3}{2}, \dfrac{1}{12} \right)$ & $\dfrac{1}{4}$  \\[0.3cm]
$\left( \left[-20,-4,12\right], 3 , \dfrac{3}{2}, \dfrac{1}{8} \right)$ & $\dfrac{1}{4}$  \\[0.3cm]
$\left( \left[-20,-4,12\right], 4 , \dfrac{3}{2}, \dfrac{1}{6} \right)$ & $\dfrac{1}{4}$  \\[0.3cm]
$\left( \left[-20,-4,12\right], 5 , \dfrac{3}{2}, \dfrac{5}{24} \right)$ & $\dfrac{1}{4}$  \\[0.3cm]
$\left( \left[-20,-4,12\right], 6 , \dfrac{3}{2}, \dfrac{1}{4} \right)$ & $\dfrac{1}{4}$  \\[0.3cm]
$\left( \left[-18,-2,14\right], 2 , \dfrac{7}{4}, \dfrac{1}{12} \right)$ & $\dfrac{1}{4}$  \\[0.3cm]
$\left( \left[-18,-2,14\right], 4 , \dfrac{7}{4}, \dfrac{1}{6} \right)$ & $\dfrac{1}{4}$  \\[0.3cm]
$\left( \left[-18,-2,14\right], 6 , \dfrac{7}{4}, \dfrac{1}{4} \right)$ & $\dfrac{1}{4}$  \\[0.3cm]
$\left( \left[-16,0,16\right], 1 , 2, \dfrac{1}{24} \right)$ & $\dfrac{1}{4}$  \\[0.3cm]
$\left( \left[-16,0,16\right], 2 , 2, \dfrac{1}{12} \right)$ & $\dfrac{1}{4}$   \\[0.3cm]
$\left( \left[-16,0,16\right], 3 , 2, \dfrac{1}{8} \right)$ & $\dfrac{1}{4}$  \\[0.3cm]
$\left( \left[-16,0,16\right], 4 , 2, \dfrac{1}{6} \right)$ & $\dfrac{1}{4}$  \\[0.3cm]
$\left( \left[-16,0,16\right], 5 , 2, \dfrac{5}{24} \right)$ & $\dfrac{1}{4}$  \\[0.3cm]
$\left( \left[-16,0,16\right], 6 , 2, \dfrac{1}{4} \right)$ & $\dfrac{1}{4}$  \\[0.3cm]
$\left( \left[-16,0,16\right], 7 , 2, \dfrac{7}{24} \right)$ & $\dfrac{1}{4}$  \\[0.3cm]
$\left( \left[-16,0,16\right], 8 , 2, \dfrac{1}{3} \right)$ & $\dfrac{1}{4}$  \\[0.3cm]
$\left( \left[-14,2,18\right], 2 , \dfrac{9}{4}, \dfrac{1}{12} \right)$ & $\dfrac{1}{4}$  \\[0.3cm]
$\left( \left[-14,2,18\right], 4 , \dfrac{9}{4}, \dfrac{1}{6} \right)$ & $\dfrac{1}{4}$  \\[0.3cm]
$\left( \left[-14,2,18\right], 6 , \dfrac{9}{4}, \dfrac{1}{4} \right)$ & $\dfrac{1}{4}$  \\[0.3cm]
$\left( \left[-12,4,20\right], 1 , \dfrac{5}{2}, \dfrac{1}{24} \right)$ & $\dfrac{1}{4}$  \\[0.3cm]
$\left( \left[-12,4,20\right], 2 , \dfrac{5}{2}, \dfrac{1}{12} \right)$ & $\dfrac{1}{4}$  \\[0.3cm]
$\left( \left[-12,4,20\right], 3 , \dfrac{5}{2}, \dfrac{1}{8} \right)$ & $\dfrac{1}{4}$  \\[0.3cm]
$\left( \left[-12,4,20\right], 4 , \dfrac{5}{2}, \dfrac{1}{6} \right)$ & $\dfrac{1}{4}$  \\[0.3cm]
$\left( \left[-12,4,20\right], 5 , \dfrac{5}{2}, \dfrac{5}{24} \right)$ & $\dfrac{1}{4}$  \\[0.3cm]
$\left( \left[-12,4,20\right], 6 , \dfrac{5}{2}, \dfrac{1}{4} \right)$ & $\dfrac{1}{4}$  \\[0.3cm]
$\left( \left[6\right], \dfrac{3}{2} , \dfrac{21}{8}, \dfrac{1}{8} \right)$ & $\dfrac{1}{4}$  \\[0.3cm]
$\left( \left[-10,6,22\right], 2 , \dfrac{11}{4}, \dfrac{1}{12} \right)$ & $\dfrac{1}{4}$  \\[0.3cm]
$\left( \left[-10,6,22\right], 4 , \dfrac{11}{4}, \dfrac{1}{6} \right)$ & $\dfrac{1}{4}$  \\[0.3cm]
$\left( \left[-8,8,24\right], 1 , 3, \dfrac{1}{24} \right)$ & $\dfrac{1}{4}$  \\[0.3cm]
$\left( \left[-8,8,24\right], 2 , 3, \dfrac{1}{12} \right)$ & $\dfrac{1}{4}$  \\[0.3cm]
$\left( \left[-8,8,24\right], 3 , 3, \dfrac{1}{8} \right)$ & $\dfrac{1}{4}$  \\[0.3cm]
$\left( \left[-8,8,24\right], 4 , 3, \dfrac{1}{6} \right)$ & $\dfrac{1}{4}$  \\[0.3cm]
$\left( \left[-6,10,26\right], 2 , \dfrac{13}{4}, \dfrac{1}{12} \right)$ & $\dfrac{1}{4}$  \\[0.3cm]
$\left( \left[-4,12,28\right], 1 , \dfrac{7}{2}, \dfrac{1}{24} \right)$ & $\dfrac{1}{4}$  \\[0.3cm]
$\left( \left[-4,12,28\right], 2 , \dfrac{7}{2}, \dfrac{1}{12} \right)$ & $\dfrac{1}{4}$  \\[0.3cm]
\hline
\label{table4}
\end{longtable}

\end{document}